\documentclass[12pt]{amsart}
\usepackage{amsfonts}

\newtheorem{theorem}{Theorem}[section]
\newtheorem{corollary}[theorem]{Corollary}

\newtheorem{lemma}[theorem]{Lemma}
\newtheorem{proposition}[theorem]{Proposition}
\newtheorem{example}[theorem]{Example}

\def\NN{\hbox{\sf I\kern-.13em\hbox{N}}}
\def\RR{\hbox{\sf I\kern-.14em\hbox{R}}}

\def\Cc{\hbox{\sf C\kern -.47em {\raise .48ex \hbox{$\scriptscriptstyle |$}}
   \kern-.5em {\raise .48ex \hbox{$\scriptscriptstyle |$}} }}

\newcommand{\be}{\begin{equation}}
\newcommand{\ee}{\end{equation}}
\newcommand{\cS}{{\mathcal S}}
\newcommand{\cF}{{\mathcal F}}
\newcommand{\cC}{{\mathcal C}}
\newcommand{\cM}{{\mathcal M}}

\newcommand{\tr}{{\rm tr} \,}
\newcommand{\rank}{{\rm rank} \,}

\begin{document}

\baselineskip 7.2mm

\title{On semigroups of nonnegative functions and positive operators}
\let\thefootnote\relax\footnote{The paper will appear in 
Journal of Mathematical Analysis and Applications, http://dx.doi.org/10.1016/j.jmaa.2013.03.003 .}
\author{Roman Drnov\v{s}ek, Heydar Radjavi}

\begin{abstract}
\baselineskip 7mm
We give extensions of results on nonnegative matrix semigroups which deduce finiteness or boundedness of such semigroups from the corresponding local properties, e.g., from finiteness or boundedness of values of certain linear functionals applied to them.
We also consider more general semigroups of functions.
\end{abstract}

\maketitle

\noindent
{\it Math. Subj. Classification (2000)}:  20M20, 15B48
 
\vspace{5mm}
\section{Introduction}
\vspace{3mm}

The general theme of this paper is extracting finiteness or boundedness information about a semigroup from hypotheses of local finiteness or 
local boundedness. Recent results in references cited below include answers to questions of the following type: 
Let $\cS$ be a (multiplicative) semigroup of nonnegative $n \times n$ matrices, i.e., members of $\cM_n (\RR^+)$. Assume $\cS$ is
indecomposable, that is, has no invariant subspace spanned by a nonempty, proper subset of the standard basis vectors.
Consider a positive linear functional $\varphi$ on $\cM_n (\RR)$. Under what conditions would finiteness or boundedness  of 
$\varphi(\cS)$ imply the same property for $\cS$ itself? Our purpose here is to give extensions of several of these results.
Also, some of our proofs are substantially shorter than the original ones.

\vspace{5mm}
\section{Entry-wise boundedness}
\vspace{3mm}

In this section we give extensions of the results in \cite{GPRS} that deduce boundedness of all entries in 
an indecomposable matrix semigroup from that of values of a given positive linear functional. 
We start with an observation that could also be useful in network and graph theory.

Let $X$ be the set of vertices in a digraph.
Define weights on the edges of the digraph by a function $\mu: X \times X \mapsto [-\infty, \infty)$, 
with the understanding that $\mu(x, y) = -\infty$ indicates that the edge from the vertex $x$ to the vertex $y$ is not present.
Loops (edges incident at both ends to the same vertex) are also allowed.
The weight of any path in the digraph is the sum of the weights of the edges in the path. 
Let $W_\mu (x, y)$ be the supremum of the weights of all paths from $x$ to $y$. This defines a function
$W_\mu: X \times X \mapsto [-\infty, \infty]$. 
{\bf Let us consider only the case when $\infty$ is not in the range of $W_\mu$.}
In this case the function $W_\mu: X \times X \mapsto [-\infty, \infty)$ obviously satisfies the inequality
\begin{equation}
\label{triangle}
W_\mu(x, y) + W_\mu(y, z) \le  W_\mu(x, z) ,
\end{equation}
under the usual extended arithmetic on $[-\infty, \infty)$.

\begin{proposition} 
\label{digraph}
(a) If there is a vertex $x_0 \in X$ such that either $W_\mu(x, x_0) \in \RR$ for all $x \in X$
or $W_\mu(x_0,x) \in \RR$ for all $x \in X$, then there exists a function $\rho : X \mapsto \RR$ such that 
\begin{equation}
\label{dominated}
W_\mu (x, y) \le \rho(x) - \rho(y)
\end{equation}
for all $x$ and $y$ in $X$. 

(b) Suppose there is a positive constant $K$ such that $W_\mu (x, y) \le K$ for all $x$ and $y$ in $X$. 
Then there exists a function $\lambda: X \times X \mapsto [-K, K]$ such that $\mu \le \lambda$ and 
$W_\mu (x, y) \le W_{\lambda} (x, y) \le K$ for all $x$ and $y$ in $X$. Consequently, 
there exists a function $\rho : X \mapsto  [-K, K]$ such that the inequality (\ref{dominated}) holds.
\end{proposition}

\begin{proof}
(a) If $W_\mu(x, x_0) \in \RR$ for all $x \in X$, then define 
$\rho(x) = W_\mu(x, x_0)$ and observe that 
$$ W_\mu(x, y) \le  W_\mu(x, x_0) -  W_\mu(y, x_0) = \rho(x) - \rho(y)$$
by the inequality (\ref{triangle}). Similarly, if $W_\mu(x_0, x) \in \RR$ for all $x \in X$, then for the function 
$\rho(x) = - W_\mu(x_0, x)$ the inequality (\ref{dominated}) holds.

(b) Define the function $\lambda$ as follows: any edge of the complete digraph on $X$ that has a $\mu$-weight below $-K$ 
(including a $\mu$-weight of $-\infty$) is bumped up to a $\lambda$-weight of $-K$, while the weights of the other edges 
are unchanged. Since any path in the complete digraph is a concatenation of alternating paths whose weights have not changed 
(and thus are at most $K$) and the paths
made up of edges that have gained new weights (all of which are $-K$), the $\lambda$-weight
of any path on $X$ is at most $K$, so that  $W_{\lambda} (x, y) \le K$ for all $x$ and $y$ in $X$.
The last assertion then follows from (a) and its proof, since
$$ W_{\lambda} (x, y) \ge  \lambda(x, y) \ge - K $$ 
for all $x$ and $y$ in $X$.
\end{proof}

Applying the exponential function to $\mu$, $W_\mu$ and $\rho$ we obtain the following multiplicative analogue of Proposition \ref{digraph}.
When applied to actual matrix semigroups, this result will be shown to yield more familiar-sounding corollaries.

\begin{corollary} 
\label{cycles}
Let $X$ be an arbitrary set, and let $f: X \times X \mapsto [0, \infty)$ be a function such that
$$ C_f(x, y) = \sup \{ f(x, x_1) f(x_1, x_2) f(x_2, x_3) \cdots f(x_k, y) : k \in \NN \cup \{0\}, \ x_1, \ldots, x_k \in X \} < \infty $$
for all $x$ and $y$ in $X$. 

(a) If there is a point $x_0 \in X$ such that either $C_f(x, x_0) > 0$ for all $x \in X$ or $C_f(x_0, x) > 0$ for all $x \in X$, 
then there exists  a function $d : X \mapsto (0, \infty)$ such that 
\begin{equation}
\label{bounded}
f(x, y) \le C_f(x, y) \le \frac{d(x)}{d(y)}
\end{equation}
for all $x$ and $y$ in $X$. 

(b) If there is a constant $M \ge 1$ such that $C_f (x, y) \le M$ for all $x$ and $y$ in $X$, then 
there exists a function $d : X \mapsto [\frac{1}{M}, M]$ such that {\rm (\ref{bounded})} holds.
\end{corollary}

Extending a definition from \cite{GPRS}, a function $f: X \times X \mapsto [0, \infty)$ is called {\it compressed} if 
$f(x, y) f(y, z) \le f(x, z)$ for all $x$, $y$ and $z$ in $X$. 

The following corollary extends Lemmas 6 and 16 of \cite{GPRS}.

\begin{corollary}
\label{compressed}
Let $M \ge 1$ be a real number and let $f: X \times X \mapsto [0, M]$ be a compressed function.
Then there exists a function $d : X \mapsto [\frac{1}{M}, M]$ such that 
$$   f(x, y) \le \frac{d(x)}{d(y)}  $$
for all $x$ and $y$ in $X$.
\end{corollary}

\begin{proof}
Clearly, we have $C_f(x,y) = f(x,y)$ for all $x$ and $y$ in $X$, and so Corollary \ref{cycles} (b) can be applied.
\end{proof}

A set $\cS$ of nonnegative functions on $X \times X$ is {\it indecomposable} if, for every $x$, $y \in X$, 
there exists $f \in \cS$ such that $f(x,y) > 0$, and it is said to be {\it bounded entry-wise} if 
$\sup \{ f(x,y) : f \in \cS \} < \infty$ for every $x$, $y \in X$. 

We now consider semigroups of functions that generalize semigroups of (not necessarily finite) matrices.
A set $\cS$ of nonnegative functions on $X \times X$, closed under a given associative operation $*$, is called a {\it matrix-like semigroup} if 
$$ (f*g) (x,z) \ge f(x,y) \, g(y,z) $$
for all $f$, $g \in \cS$ and $x$, $y$, $z \in X$. 

\begin{lemma}
\label{bounded entry-wise}
Let $\cS$ be a matrix-like semigroup of functions on $X \times X$ that is bounded entry-wise.
Then the function $s(x,y) = \sup \{ f(x,y) : f \in \cS \}$ is compressed.
\end{lemma}

\begin{proof}
For all $f$, $g \in \cS$, we have 
$f(x,y) g(y,z) \le (f*g)(x,z) \le s(x,z)$
for all $x$, $y$ and $z$ in $X$. It follows that 
$s(x,y) s(y,z) \le s(x,z)$, as asserted.
\end{proof}

The following theorem is an extension of \cite[Theorem 17]{GPRS}.

\begin{theorem}
\label{bounded_functions}
Let $M \ge 1$ be a real number and let $\cS$ be a matrix-like semigroup of functions from $X \times X$ to $[0, M]$.
Then there exists a function $d : X \mapsto [\frac{1}{M}, M]$ such that 
$$ f(x, y) \le \frac{d(x)}{d(y)} $$
for all $f \in \cS$ and $x$, $y \in X$. 
\end{theorem}

\begin{proof}
By Lemma \ref{bounded entry-wise}, the function $s(x,y) = \sup \{ f(x,y) : f \in \cS \}$ is compressed.
Since it maps to $[0, M]$, the conclusion of the theorem follows from Corollary \ref{compressed}.
\end{proof}

\begin{corollary}
Let $M \ge 1$ be a real number and let $(X, \mu)$ be an atomic $\sigma$-finite measure space with the property 
that the measure of each atom is at least $1$.  
Let $\cS$ be a semigroup of integral operators on $L^2(X,  \mu)$
whose integral kernels map to the interval $[0, M]$. Then there exists a function $d : X \mapsto [\frac{1}{M}, M]$ 
such that, for the multiplication operator $D$ on $L^2(X,  \mu)$ induced by $d$, the integral kernel of $D^{-1} S D$ maps to 
$[0, 1]$ for every $S \in \cS$. 
\end{corollary}

\begin{proof}
The set of all integral kernels $\{k_S : S \in \cS\}$ of operators in $\cS$ is 
a matrix-like semigroup of functions from $X \times X$ to $[0, M]$ with the semigroup operation
$k_S * k_T = k_{S T}$. Indeed, we have 
$$ k_{S T}(x, z) = \int_X \! k_S(x,u) k_T(u,z) d\mu(u) \geq k_S(x,y) k_T(y,z) \mu(\{A_{y}\}) \ge k_S(x,y) k_T(y,z) $$
for all $S$, $T \in \cS$ and $x$, $y$, $z \in X$. 
Here $A_{y}$ denotes the atom containing the point $y$.

By Theorem \ref{bounded_functions}, there exists a function $d : X \mapsto [\frac{1}{M}, M]$ such that 
$$ k_S (x, y) \le \frac{d(x)}{d(y)} $$
for all $S \in \cS$ and $x$, $y \in X$, so that  
$$ k_{D^{-1} S D} (x, y) =  \frac{1}{d(x)} k_S (x, y) d(y) \in [0, 1] . $$
\end{proof}

The next theorem is an extension of \cite[Theorem 19]{GPRS}.

\begin{theorem}
\label{bounded_semigroup}
Let $\cS$ be an indecomposable matrix-like semigroup of nonnegative functions on $X \times X$.
If there exist $u$, $v \in X$ such that $\sup \{ f(u,v) : f \in \cS \} < \infty$,
then there exists a function $d : X \mapsto (0, \infty)$ such that 
$$ f(x, y) \le \frac{d(x)}{d(y)} $$
for all $f \in \cS$ and $x$, $y \in X$.  
\end{theorem}

\begin{proof}
First, we claim that $\cS$ is bounded entry-wise. Define $M = \sup\{f(u,v) : f \in \cS\}$, and choose any 
$x$, $y \in X$, and $f \in \cS$. 
Since $\cS$ is indecomposable, there exist $g$, $h \in \cS$ such that $g(u,x) > 0$ and $h(y,v) > 0$. 
Then 
$$ M \ge (g * f * h)(u,v) \ge g(u,x) f(x,y) h(y,v) , $$
and so 
$$ f(x,y) \le \frac{M}{g(u,x) h(y,v)} . $$
This proves the claim.

Now,  the function $s(x,y) = \sup\{ f(x,y) : f \in \cS \}$ is compressed by Lemma \ref{bounded entry-wise}.
Since it maps to $(0, \infty)$, Corollary \ref{cycles}(a) can be applied to complete the proof.
\end{proof}

\vspace{5mm}
\section{Binary diagonals}
\vspace{3mm}

The main result of \cite{LMR} is the following theorem. 
We recall that a square matrix is said to have a {\it binary diagonal} if 
its diagonal entries all come from the set $\{0, 1\}$. Furthermore, a square matrix is {\it binary} if 
its entries come from the set $\{0, 1\}$.

\begin{theorem} \cite{LMR}
Every indecomposable semigroup of nonnegative matrices with binary diagonals 
is up to a similarity a semigroup of binary matrices. Moreover, the similarity can be implemented
by an invertible, positive, diagonal matrix.
\end{theorem}

We now extend this result to our setting. We note that the proof presented below is much shorter than the proof in \cite{LMR}.
A nonnegative function $f$ on $X \times X$ is said to have a {\it binary diagonal} if 
$f(x,x) \in \{0, 1\}$ for all $x \in X$. 

\begin{theorem}
Let $\cS$ be an indecomposable semigroup of nonnegative functions on $X \times X$,
where the multiplication of $f$ and $g$ in $\cS$ is defined by
$$ (f*g)(x, y) = \sum_{z \in X} f(x,z) g(z,y) . $$
(Here the finiteness of the sum of nonnegative numbers is part of the hypothesis.)
If every function $f \in \cS$ has a binary diagonal, then there exists a function $d : X \mapsto (0, \infty)$ such that 
$$  \frac{f(x, y) d(y)}{d(x)} \in \left\{ 0, 1 \right\} $$
for all $f \in \cS$ and all $x, y \in X$. 
\end{theorem}

\begin{proof}
Clearly, we may assume that $\cS$ is maximal with respect to the inclusion.
Then $\cS$ necessarily contains the characteristic function of the diagonal of $X \times X$, 
which of course acts as an identity with respect to $*$. 
Given $u \in X$, let $e_u$ denote the characteristic function of $\{(u,u)\}$.  
We will prove that $e_u \in \cS$ for each $u \in X$.

First, we claim that $f * e_u * g$ has a binary diagonal for all $f, g \in \cS$, i.e.,  
$$ (f * e_u * g)(x,x) = f(x,u) g(u,x) \in \{0, 1\} $$ 
for all $x \in X$.  Since this holds for $x = u$, we assume that $0 < f(x,u) g(u,x) \neq 1$ for some $x \neq u$.
Then 
$$ (f*g) (x,x) = \sum_{y \in X} f(x,y) g(y,x)  \ge  f(x,u) g(u,x) > 0 , $$
so that $(f*g) (x,x) = 1$ and $f(x,u) g(u,x) < 1$. It follows that there exists $y \notin \{u, x\}$ such that  $f(x,y) g(y,x) \in (0, 1)$. 
Since  
$$ (g*f) (u,u) = \sum_{z \in X} g(u,z) f(z,u)  \ge  f(x,u) g(u,x) > 0 , $$
we have $(g*f)(u,u) = 1$.
Now observe that $(g*f)(u,y) \geq g(u,x) f(x,y) > 0$ and $(g*f)(y,u) \geq g(y,x) f(x,u) > 0$, so that 
$$ (g*f*g*f)(u,u) \geq ((g*f)(u,u))^2 +  (g*f)(u,y)(g*f)(y,u) > 1 . $$
This contradiction proves the claim. 

To prove that $e_u \in \cS$, we take any functions $f_1$, $f_2$, $\ldots$, $f_n$ in $\cS$ , and we
observe that, for each $x \in X$,
$$ (f_1 * e_u *  f_2 * e_u * \ldots * e_u * f_n)(x,x) = f_2(u,u) f_3(u,u)  f_{n-1}(u,u) 
( f_1 * e_u * f_n)(x,x) \in \{0, 1\} $$  
by the above. Therefore, the maximality of $\cS$ implies that  $e_u \in \cS$.

By Theorem \ref{bounded_semigroup}, there exists a function $d : X \mapsto (0, \infty)$ such that 
$$ f(x, y) \le \frac{d(x)}{d(y)} $$ 
for all $f \in \cS$ and $x$, $y \in X$.  
It remains to show that $f(x, y) = \frac{d(x)}{d(y)}$ provided $f(x,y) > 0$. 
Since $\cS$ is indecomposable, there exists $g  \in \cS$ such that $g(y,x) > 0$.  Since 
$$ (f*e_y*g)(x,x) = f(x,y) g(y,x) > 0 , $$
we have $ (f*e_y*g)(x,x) = 1$, so that $f(x,y) g(y,x) = 1$. 
However,  $f(x,y) \le  \frac{d(x)}{d(y)}$ and  $g(y,x) \le  \frac{d(y)}{d(x)}$,
and so we must have that $f(x,y)  =  \frac{d(x)}{d(y)}$.
\end{proof}

\vspace{5mm}
\section{Finite diagonals and finite traces}
\vspace{3mm}

We first recall one of the main results of \cite{PRW}. A semigroup $\cS$ of complex matrices is said to have 
 {\it finite diagonals} if all the diagonal entries of all the matrices in $\cS$ come
from a finite set. A collection $\cC$ of matrices is called {\it self-adjoint} if for each $T \in \cC$ we have $T^* \in \cC$. 
Here $T^*$ is just the conjugate transpose of $T$.

\begin{theorem}  \cite{PRW}
\label{finite_diagonals}
Suppose that a semigroup $\cS$ of nonnegative matrices has finite diagonals.
If $\cS$ is self-adjoint, then it is finite. Moreover, nonzero entries of matrices in $\cS$ are of the form
$\sqrt{\xi \eta}$, where $\xi$ and $\eta$ are diagonal values of some matrices in $\cS$.
\end{theorem}

We also recall one of the assertions of \cite[Theorem 8]{RR}. We state a stronger conclusion 
whose proof is similar to that of \cite[Theorem 8]{RR}.

\begin{theorem} \cite{RR}
If the range of the trace is finite on a semigroup $\cS$ of complex matrices, then 
$\cS$ is unitarily similar to a block-triangular semigroup whose diagonal blocks come from a finite set of matrices.
\end{theorem}

Observe that the preceding theorem implies the following corollary that also gives 
the first assertion of Theorem \ref{finite_diagonals}. 

\begin{corollary}
\label{finite_trace}
If the range of the trace is finite on a self-adjoint semigroup $\cS$ of complex matrices, 
then $\cS$ is finite.
\end{corollary}

\begin{proof}
The semigroup $\cS$ unitarily similar to a block-diagonal semigroup whose diagonal blocks come from a finite set.
\end{proof}

We now give infinite-dimensional extensions of these results. We restrict our general setting to semigroups of (bounded) operators 
on the real or complex Hilbert space $l^2$. A collection $\cC$ of operators on $l^2$ is called {\it self-adjoint} 
if for each $T \in \cC$ we have $T^* \in \cC$. An operator $T$ on $l^2$ can be represented 
by an infinite real or complex matrix $(T_{i j})_{i, j \in \NN}$ with respect to the standard basis of $l^2$, 
and so the matrix of $T^*$ is just the conjugate transpose of the matrix of $T$.  

For a semigroup $\cS$ of operators on $l^2$, we denote by $\cS_+$ the set of all positive semidefinite operators in $\cS$,
and by $P(\cS)$ the set of all projections (self-adjoint idempotents) in $\cS$.
Clearly,  $P(\cS) \subseteq \cS_+ \subseteq \cS$. 

\begin{lemma}
\label{projections}
Let $\cS$ be a self-adjoint semigroup of operators on $l^2$. Suppose that either:

(i) $\cS$ consists of trace-class operators and the set $\{ \tr S : S \in \cS_+ \}$ is finite, or 

(ii) for each $i \in \NN$ the set $\cF_i = \{ S_{i i} : S \in \cS_+ \}$ is finite. 

\noindent
Then the following hold:

(a) $S S^*$ is a projection for every $S \in \cS$;

(b) Every nonzero $S \in \cS$ is a partial isometry, and so $\|S\| = 1$.

(c) Every idempotent in $\cS$ is a projection;

(d) The set $P(\cS)$ is commutative (and thus it is a subsemigroup of $\cS$);

(e) In the case (i) every member of $\cS$ is of finite rank not exceeding 
$$ r := \max \{ \tr S : S \in \cS_+ \} . $$
\end{lemma}

\begin{proof}
(a) Given $S \in \cS$, the positive semidefinite operator $P = S S^*$ belongs to $\cS_+$.
Consider first the case (i). Since the nonzero eigenvalues of $P$ are all positive and the set $\{ \tr (P^n) : n \in \NN \}$ is finite, 
we conclude that the spectrum of $P$ is contained in $\{0, 1\}$, and thus $P$ is a projection.

For the case (ii), let $P = \int_{(0, \infty)} \! t \, dE(t)$ be the spectral representation of $P$, 
and let $\{e_i\}_{i \in \NN}$ be the standard basis of $l^2$. 
For each $i \in \NN$ we define the scalar Borel measure $\mu_i$ by $\mu_i(B) = \langle E(B) e_i, e_i \rangle$.
Then, for each $k \in \NN$, 
$$ \langle P^k e_i, e_i \rangle = \int_{(0, \infty)} \! t^k \, d\mu_i(t) = 
\int_{(0, 1)} \! t^k \, d\mu_i(t) + \int_{(1, \infty)} \! t^k \, d\mu_i(t) + \mu_i(\{1\}) . $$
Since $\langle P^k e_i, e_i \rangle \in \cF_i$ and the set $\cF_i$ is finite, we conclude that $\mu_i((0, 1) \cup (1, \infty)) = 0$.
It follows that $E((0, 1) \cup (1, \infty)) e_i = 0$ for all $i \in \NN$, and so $P = E(\{1\})$. This completes the proof of (a).

(b) This follows from (a).

(c) If $0 \neq E \in \cS$ is an idempotent, then $\|E\| = 1$ by (b), implying $E^* = E$.

(d) Let $P$ and $Q$ be projections in $\cS$. Then $PQP = PQ (PQ)^*$ is a projection in $\cS$ by (a). But then 
$$ [PQ(I-P)] [PQ(I-P)]^* = PQ (I-P)QP = PQP - (PQP)^2 = 0  $$
implying $PQ = PQP$, which means $PQ = (PQ)^* = QP$.

(e) Since $S S^*$ is a projection for every $S \in \cS$ by (a), we have $\rank (S S^*) = \tr (S S^*)$, and so $S S^*$ has finite rank not exceeding $r$.
Now observe that $S$ and $S S^*$ have the same range.
\end{proof}

\begin{theorem}
Let $\cS$ be a self-adjoint semigroup of trace-class operators on $l^2$ such that the set $\{ \tr S : S \in \cS \}$ is finite. 
If $\cS$ contains only finitely many projections, then $\cS$ is finite.
\end{theorem}

\begin{proof}
By Lemma \ref{projections}, the projections of $\cS$ commute and are all finite-rank.  
Thus, there is a positive integer $n$ such that all the projections of $\cS$ have ranges contained in the same fixed $n$-dimensional subspace.  
Since the range of every member of $\cS$ is also the range of some projection in $\cS$, 
the whole semigroup $\cS$ is contained in the direct sum of $\cM_n (\Cc)$ and a zero block.
Now we apply Corollary \ref{finite_trace}. 
\end{proof}
 
In the following theorem we impose the additional assumption that operators are positive in the sense employed for maps on 
the Banach lattice $l^2$. 
 
\begin{theorem}
\label{extension_finite_diagonals}
Let $\cS$ be a self-adjoint semigroup of positive operators on the Banach lattice $l^2$. Suppose that either:

(i) $\cS$ consists of trace-class operators and the set $\{ \tr S : S \in \cS_+ \}$ is finite, or 

(ii) for each $i \in \NN$ the set $\{ S_{i i} : S \in \cS_+ \}$ is finite. 

\noindent
Then the nonzero entries of $S \in \cS$ are of the form $\sqrt{\xi \eta}$, where $\xi$ and $\eta$ 
are the diagonal entries of the projections $S S^*$ and $S^* S$, respectively. 
\end{theorem}

\begin{proof}
Given $S \in \cS$, the operator $P = S S^*$ belongs to $P(\cS)$ by Lemma \ref{projections}.
Since $P$ is also a positive operator, it has the following form: up to a permutation similarity, 
there are (finitely or infinitely many) strictly positive vectors $\{x_k\}_k$ of finite or infinite length such that $x_k^* x_k = 1$ for all $k$ 
and the matrix $(P_{i j})_{i, j \in \NN}$ of $P$ is block diagonal with blocks $\{x_k x_k^*\}$ and a (possibly) zero block 
(see e.g. \cite[Lemma 5.1.9 or Lemma 8.7.12]{RaRo}). 
This means that $(P_{i j})^2 = P_{i i} \, P_{j j}$ whenever $P_{i j} \neq 0$. 
Therefore, we have $((S S^*)_{i j})^2 = (S S^*)_{i i} \, (S S^*)_{j j}$ or
$$ \left( \sum_{k=1}^{\infty} S_{i k} S_{j k} \right)^2 = \left( \sum_{k=1}^{\infty} S_{i k}^2 \right) 
   \left( \sum_{k=1}^{\infty} S_{j k}^2 \right) $$
if $S_{i k} \, S_{j k} \neq 0$ for some $k$. Hence, in this case the equality holds in the Cauchy-Schwarz inequality, so that 
the $j$-th row of the matrix of $S$ is a multiple of the $i$-the row. This implies that the matrix of $S$ must have the same form as in the finite-dimensional case (see the representation (3) in \cite[p. 1416] {PRW}): 
$$ \Delta_1 \left[ \begin{array}{cccc}
u_1 v_1^* & 0         &  0        & \ldots  \\
   0      & u_2 v_2^* &  0        & \ldots  \\
   0      &   0       & u_3 v_3^* & \ldots  \\
\vdots & \vdots &  \vdots  &  \ddots 
\end{array} \right]  \Delta_2^* , $$
where $\Delta_1$ and $\Delta_2$ are (infinite) permutation matrices, and for each $k$ the vectors $u_k$ and $v_k$ are both either strictly positive 
or zero (the rectangular blocks $u_k v_k^*$ are in general not square and there may only be one block).

Considering the projections $S S^*$ and $S^* S$ we conclude that $(u_k^* u_k) (v_k^* v_k) \in \{0, 1\}$, 
and the nonzero entries of $S$ are of the form $\sqrt{\xi \eta}$, 
where $\xi$ and $\eta$ are the diagonal entries of the projections $S S^*$ and $S^* S$, respectively. 
\end{proof}

 
It is easy to see that Theorem \ref{extension_finite_diagonals} implies Theorem \ref{finite_diagonals}. 
In fact, Theorem \ref{extension_finite_diagonals} implies the following finite-dimensional generalization of Theorem \ref{finite_diagonals}.

\begin{corollary}
Let $\cS$ be a  self-adjoint semigroup of nonnegative matrices such that the set $\{ \tr S : S \in \cS_+ \}$ is finite.
Then nonzero entries of matrices in $\cS$ are of the form $\sqrt{\xi \eta}$, 
where $\xi$ and $\eta$ are diagonal values of some matrices in $P(\cS)$. In particular, $\cS$ is finite.
\end{corollary}

\begin{proof}
The first assertion follows from Theorem \ref{extension_finite_diagonals}. 
To prove that $\cS$ is finite, just note that the commutative semigroup $P(\cS)$ of all projections in $\cS$ 
is unitarily similar to the semigroup of diagonal matrices with binary diagonals that is clearly a finite semigroup.
\end{proof}

It is tempting to conjecture that the condition (i) of Theorem \ref{extension_finite_diagonals} implies finiteness of the semigroup or at least the condition (ii). We conclude the paper with a counterexample. 

\begin{example} {\rm
Let $c = \frac{1}{\sqrt{2}}$ and $f = (c, c^2, c^3, \ldots) \in l^2$. For a positive integer $m$, let $g_m$ denote the vector obtained from 
$f$ by annihilating alternate segments of length $2^m$, i.e., 
$$ g_m = (\underbrace{c, c^2, \ldots, c^{2^m}}_{2^m}, \underbrace{0, 0, \ldots, 0}_{2^m}, 
\underbrace{c^{2^{m+1}+1}, c^{2^{m+1}+2}, \ldots, c^{3 \cdot 2^{m}}}_{2^m}, \underbrace{0, 0, \ldots, 0}_{2^m}, \ldots ) , $$
and let $h_m = f - g_m$, so that 
$$ h_m = (\underbrace{0, 0, \ldots, 0}_{2^m}, \underbrace{c^{2^{m}+1}, c^{2^{m}+2}, \ldots, c^{2^{m+1}}}_{2^m}, 
\underbrace{0, 0, \ldots, 0}_{2^m}, \underbrace{c^{3 \cdot 2^{m}+1}, c^{3 \cdot 2^{m}+2}, \ldots, c^{2^{m+2}}}_{2^m}, \ldots ) . $$
Then the operator 
$$ Q_m = \frac{g_m g_m^*}{\|g_m\|^2} + \frac{h_m h_m^*}{\|h_m\|^2}  $$
is a projection on $l^2$ of rank two. Defining also the rank-one projection $P = f f^*$ we now claim that the set 
$$ \cS = \{P, Q_1, Q_2, Q_3, \ldots\} $$ 
is a semigroup.

Since $f^* g_m = g_m^* g_m + h_m^* g_m = g_m^* g_m  = \|g_m\|^2$ and (similarly) $f^* h_m = \|h_m\|^2$, we have 
$$ P Q_m = f g_m^* + f h_m^* = P \ \ \ \textrm{and} \ \ \ \ Q_m P = (P Q_m)^* = P . $$
Now choose positive integers $m$ and $n$ such that $m < n$. We will show that $Q_m Q_n = P$, and then we also have
$Q_n Q_m = (Q_m Q_n)^* = P$ completing the proof of the claim.
A glance at the definitions of $g_m$ and $h_m$ gives that $\|h_m\| = c^{2^m} \|g_m\|$. Since $\|g_m\|^2 + \|h_m\|^2 = \|f\|^2 = 1$, we conclude that 
$$ \|g_m\|^2 = \left( c^{2^{m+1}}+1 \right)^{-1} \ \ \ \textrm{and} \ \ \ \  \|h_m\|^2 = c^{2^{m+1}} \left( c^{2^{m+1}}+1 \right)^{-1} . $$
Another look at the definitions reveals that $h_m^* g_n = c^{2^{m+1}} g_m^* g_n$, and so,
since $g_m^* g_n + h_m^* g_n = f^* g_n = \|g_n\|^2$, we have 
$$ g_m^* g_n = \left( c^{2^{m+1}}+1 \right)^{-1} \|g_n\|^2 = \|g_m\|^2 \|g_n\|^2  \ \ \ \ \textrm{and} $$
$$ h_m^* g_n = c^{2^{m+1}} \left( c^{2^{m+1}}+1 \right)^{-1} \|g_n\|^2 = \|h_m\|^2 \|g_n\|^2 . $$
Similarly, it follows from 
$h_m^* h_n = c^{2^{m+1}} g_m^* h_n$ \, and \, $g_m^* h_n + h_m^* h_n = f^* h_n = \|h_n\|^2$ that 
$$ g_m^* h_n = \left( c^{2^{m+1}}+1 \right)^{-1} \|h_n\|^2 = \|g_m\|^2 \|h_n\|^2  \ \ \ \textrm{and} $$
$$ h_m^* h_n = c^{2^{m+1}} \left( c^{2^{m+1}}+1 \right)^{-1} \|h_n\|^2 = \|h_m\|^2 \|h_n\|^2 . $$
Applying the last equalities we obtain that 
$$ Q_m Q_n = \frac{g_m^*g_n}{\|g_m\|^2 \|g_n\|^2} \, g_m g_n^* + \frac{g_m^*h_n}{\|g_m\|^2 \|h_n\|^2} \, g_m h_n^* + 
\frac{h_m^*g_n}{\|h_m\|^2 \|g_n\|^2} \, h_m g_n^* + \frac{h_m^*h_n}{\|h_m\|^2 \|h_n\|^2} \, h_m h_n^*  = $$
$$ = g_m g_n^* + g_m h_n^* + h_m g_n^* +  h_m h_n^* = (g_m + h_m) (g_n^* + h_n^*) = f f^* = P . $$
Observe that the set $\cF_i$ of all $(i,i)$ slots of members in $P(\cS) = \cS_+ = \cS$ is infinite, for example
$$ \cF_1 = \{c^2\} \cup \{ c^2 (c^{2^{m+1}}+1) : m \in \NN \} . $$
}
\end{example}
  
{\it Acknowledgments.} The authors were supported, respectively, by the Slovenian Research Agency and the NSERC of Canada. 
They express deep gratitude to the referee who improved the original version of Corollary \ref{cycles}
by stating its additive analogue in Proposition \ref{digraph}.

\vspace{3mm}

\noindent
Roman Drnov\v sek \\
Department of Mathematics, Faculty of Mathematics and Physics \\
University of Ljubljana \\
Jadranska 19, SI-1000 Ljubljana, Slovenia \\
e-mail : roman.drnovsek@fmf.uni-lj.si \\

\noindent
Heydar Radjavi \\
Department of Pure Mathematics \\
University of Waterloo \\
Waterloo, Ontario, Canada N2L 3G1  \\
e-mail : hradjavi@uwaterloo.ca

\end{document}